\documentclass{article}

\usepackage{amsmath,mathrsfs,amsthm,bbm}
\usepackage{txfonts}
\usepackage[all]{xy}

\DeclareMathOperator{\End}{\mathrm{End}}
\DeclareMathOperator{\Aut}{\mathrm{Aut}}

\DeclareMathOperator{\Sp}{\mathrm{Sp}}
\DeclareMathOperator{\SO}{\mathrm{SO}}
\DeclareMathOperator{\GL}{\mathrm{GL}}
\DeclareMathOperator{\SL}{\mathrm{SL}}
\DeclareMathOperator{\tr}{\mathrm{tr}}
\DeclareMathOperator{\OG}{\mathrm{OG}}
\DeclareMathOperator{\IG}{\mathrm{IG}}
\DeclareMathOperator{\Fl}{\mathrm{Fl}}
\newtheorem{theorem}{Theorem}
\newtheorem{lemma}[theorem]{Lemma}
\newtheorem{proposition}[theorem]{Proposition}
\newtheorem{corollary}[theorem]{Corollary}

\theoremstyle{definition}
\newtheorem{definition}[theorem]{Definition}

\newtheorem{remark}[theorem]{Remark}

\DeclareMathOperator{\poly}{\mathrm{poly}}
\DeclareMathOperator{\Rep}{\mathrm{Rep}}
\DeclareMathOperator{\Gr}{\mathrm{Gr}}
\DeclareMathOperator{\diag}{\mathrm{diag}}

\DeclareMathOperator{\sym}{\mathrm{sym}}

 \newcommand{\C}{\mathbb{C}}
 \newcommand{\Z}{\mathbb{Z}}
\newcommand{\cl}{\mathcal{L}}

\newcommand{\frt}{\mathfrak{t}}
\begin{document}

\title{Representation ring of Levi subgroups versus cohomology ring of flag varieties II}

\author{Shrawan Kumar and Sean Rogers}
\maketitle

\begin{abstract} For any reductive group $G$ and a parabolic subgroup $P$ with its Levi subgroup $L$, the first author in [Ku2] introduced a ring homomorphism 
$ \xi^P_\lambda:  \Rep^\C_{\lambda-\poly}(L) \to H^*(G/P, \C)$, where  $ \Rep^\C_{\lambda-\poly}(L)$ is a certain subring of the complexified representation ring of $L$ (depending upon the choice of an irreducible representation $V(\lambda)$ of $G$ with highest weight $\lambda$). In this paper we study this homomorphism for $G=\Sp(2n)$ and its maximal parabolic subgroups $P_{n-k}$ for any $1\leq k\leq n$ (with the choice of $V(\lambda) $ to be the defining representation  $V(\omega_1) $ in $\mathbb{C}^{2n}$). Thus, we  obtain a $\C$-algebra homomorphism  $ \xi_{n,k}:  \Rep^\C_{\omega_1-\poly}(\Sp(2k)) \to H^*(IG(n-k, 2n), \C)$. Our main result asserts that $ \xi_{n,k}$ is injective when $n$ tends to $\infty$ keeping $k$ fixed. Similar results are obtained for the odd orthogonal groups.
\end{abstract}

\section{Introduction}

This is a follow-up of first author's work [Ku2].

Let $G$ be a connected reductive group  over $\mathbb{C}$ with a Borel subgroup $B$ and maximal torus $T\subset B$. Let $P$ be a standard parabolic subgroup with the Levi subgroup $L$ containing $T$. 
Let $V(\lambda)$ be an irreducible almost faithful representation of $G$ with highest weight $\lambda$ (i.e., the corresponding map $\rho_\lambda: G \to \Aut (V(\lambda))$ has finite kernel). Then, 
Springer defined an adjoint-equivariant regular map with Zariski dense image
$
\theta_{\lambda}:G\to \mathfrak{g}$ (depending upon $\lambda$)
 (cf. Definition \ref{def1}). Using this the first author defined in [Ku2] a certain subring  $\Rep^\C_{\lambda-\poly}(L)$ of the complexified representation ring  $\Rep^\C(L)$ (cf. Definition \ref{maindefi}). 
For $G=\GL (n)$ and $V(\lambda)$ the defining representation $\C^n$, the ring  $\Rep_{\lambda-\poly}(G):=  \Rep^\C_{\lambda-\poly}(G)\cap  \Rep (G)$
coincides with the standard notion of polynomial representation ring of $\GL (n)$ (cf. the equation \eqref{eqn1'}). 

 Coming back to the general case, the first author [Ku2] defined a 
 surjective $\C$-algebra homomorphism 
\begin{equation}\label{eqnI1} \xi^P_\lambda:  \Rep^\C_{\lambda-\poly}(L) \to H^*(G/P, \C)\end{equation}
 (cf. Theorem \ref{thmmain}). 

Specializing the above result to the case when $G=\GL (n)$, $V(\lambda)$  is the standard defining representation $\C^n$  and $P=P_r$ (for any $1\leq r \leq n-1$) is the maximal parabolic subgroup so that the flag variety 
$G/P_r$ is the Grassmannian $\Gr(r,n)$ and $L_r=\GL(r)\times \GL(n-r)$ and restricting $\xi^P_\lambda$ to the component $\GL(r)$, one recovers the  classical  ring homomorphism 
$$\phi_n: \Rep_{\poly}(\GL (r)) \to H^*(\Gr(r, n))$$
as shown in [Ku2, $\S$5].

Fix $r\geq 1$ and define the stable cohomology ring 
  $$
  \mathbb{H}^* (\Gr_r, \mathbb{Z}) := \varprojlim H^* (\Gr(r, n), \mathbb{Z})
  $$
  as the inverse limit. Then, the homomorphisms $\phi_n$ combine to give  a ring homomorphism 
$$ \phi_\infty: \Rep_{\poly}(\GL (r)) \to   \mathbb{H}^* (\Gr_r, \mathbb{Z}).$$
Moreover, by the explicit description of $\phi_n$ (cf. [F, $\S$9.4] and also [Ku2, $\S$5]) it is immediately seen that $ \phi_\infty$ is a ring {\it isomorphism}. 

{\it The aim of this paper is to analyze the corresponding question for the Symplectic groups $\Sp(2k)$ as well as the odd orthogonal groups 
$\SO(2k+1)$.}

Let us fix a positive integer $k$ and consider  the isotropic Grassmannian $\IG(n-k,2n)$ consisting of 
 $n-k$-dimensional isotropic
subspaces of $V=\mathbb{C}^{2n}$ with respect to a non-degenerate  symplectic  form.
 Then, $\IG(n-k,2n)$  is the quotient $\Sp (2n)/P_{n-k}^C$ of $\Sp (2n)$ by
 the standard maximal parabolic subgroup
$P_{n-k}^C$  corresponding to the $n-k$-th node of the Dynkin diagram of $\Sp(2n)$ (following the indexing convention as in [Bo]).
Let  $L_{n-k}^C$  denote the Levi  subgroup of $P_{n-k}^C$.
Then, 
$$L_{n-k}^C\simeq \GL (n-k)\times \Sp (2k).$$
We take the standard representation of $\Sp(2n)$ in $\mathbb{C}^{2n}$ and abbreviate the corresponding  $\Rep^\C_{\lambda-\poly}(L_{n-k}^C)$ by 
$ \Rep^\C_{\poly}(L_{n-k}^C)$. Thus, following \eqref{eqnI1}, we get a ring homomorphism 
$$   \xi^{P_{n-k}^C}:  \Rep^\C_{\poly}(L_{n-k}^C) \to H^*(\IG(n-k, 2n), \C).$$
Restricting $   \xi^{P_{n-k}^C}$ to the component $\Sp (2k)$, we get a ring homomorphism
$$   \xi_{n,k}:  \Rep^\C_{\poly}( \Sp (2k)) \to H^*(\IG(n-k, 2n), \C).$$
Define the {\it stable cohomology ring} (cf. Definition \ref{defi14})
  $$
  \mathbb{H}^* (\IG_k, \mathbb{Z}) := \varprojlim H^* (\IG(n-k, 2n), \mathbb{Z})
  $$
  as the inverse limit. Then, the homomorphisms $\xi_{n,k}$ combine to give  a ring homomorphism 
$$ \xi_k:  \Rep^\C_{\poly}( \Sp (2k)) \to    \mathbb{H}^* (\IG_k, \mathbb{Z}).$$

Following is our first main result of the paper (cf. Theorem \ref{thmsp} for a more precise assertion).

\noindent
{\bf Theorem A.}
 {\it The above ring homomorphism  $\xi_k: \Rep^{\mathbb{C}}_{\poly} (\Sp (2k)) \to \mathbb{H}^* (\IG_k, \C)$  is injective.

However, it is {\it not} surjective (cf. Remark \ref{remark1}).}

There are parallel results for the odd orthogonal groups $\SO(2k+1)$. Specifically, 
consider  the isotropic Grassmannian $\OG(n-k,2n+1)$ consisting of 
 $n-k$-dimensional isotropic
subspaces of $V=\mathbb{C}^{2n+1}$ with respect to a non-degenerate  symmetric  form.
 Then, $\OG(n-k,2n+1)$  is the quotient $\SO (2n)/P_{n-k}^B$ of $\SO (2n+1)$ by
 the standard maximal parabolic subgroup
$P_{n-k}^B$  corresponding to the $n-k$-th node of the Dynkin diagram of $\SO(2n+1)$.
Let  $L_{n-k}^B$  denote the Levi  subgroup of $P_{n-k}^B$.
Then, 
$$L_{n-k}^B \simeq \GL (n-k)\times \SO (2k+1).$$
We take the standard representation of $\SO(2n+1)$ in $\mathbb{C}^{2n+1}$ and abbreviate the corresponding  $\Rep^\C_{\lambda-\poly}(L_{n-k}^B)$ by 
$ \Rep^\C_{\poly}(L_{n-k}^B)$. Thus, following \eqref{eqnI1}, we get a ring homomorphism 
$$   \xi^{P_{n-k}^B}:  \Rep^\C_{\poly}(L_{n-k}^B) \to H^*(\OG(n-k, 2n+1), \C).$$
Restricting $   \xi^{P_{n-k}^B}$ to the component $\SO (2k+1)$, we get a ring homomorphism
$$  \bar{\xi}_{n,k}:  \Rep^\C_{\poly}( \SO (2k+1)) \to H^*(\OG(n-k, 2n+1), \C).$$
Similar to $\mathbb{H}^* (\IG_k, \mathbb{Z})$, define the {\it stable cohomology ring} (cf. Definition \ref{defi14'})
  $$
  \mathbb{H}^* (\OG_k, \mathbb{Z}) := \varprojlim H^* (\OG(n-k, 2n+1), \mathbb{Z})
  $$
  as the inverse limit. Then, the homomorphisms $\bar{\xi}_{n,k}$ combine to give  a ring homomorphism 
$$ \bar{\xi}_k:  \Rep^\C_{\poly}( \SO (2k+1)) \to    \mathbb{H}^* (\OG_k, \mathbb{Z}).$$

Following is our second main result of the paper (cf. Theorem \ref{thmsp'}  for a more precise assertion).

\noindent
{\bf Theorem B.} 
 {\it The above ring homomorphism  $\bar{\xi}_k: \Rep^{\mathbb{C}}_{\poly} (\SO (2k+1)) \to \mathbb{H}^* (\OG_k, \C)$  is injective.

However, it is {\it not} surjective (cf. Remark \ref{remark2}).}

The proofs rely on some results of Buch-Kresch-Tamvakis from [BKT1] and [BKT2] and earlier results of the first author [Ku2]. 
\vskip2ex

\noindent
{\bf Acknowledgements:} We thank L. Mihalcea for providing a simple proof of Proposition \ref{chap1-thm1.5}. This work was partially supported by the NSF grant DMS-1802328.

\section{Prelimanaries and Notation}

We recall some notation and results from [Ku2]. 

Let $G$ be a connected reductive group  over $\mathbb{C}$ with a Borel subgroup $B$ and maximal torus $T\subset B$. Let $P$ be a standard parabolic subgroup with the Levi subgroup $L$ containing $T$. We denote their Lie algebras by the corresponding Gothic characters: $\mathfrak{g}, \mathfrak{b}, \mathfrak{t},\mathfrak{p}, \mathfrak{l}$
respectively. We denote by $\Delta=\{\alpha_1, \dots, \alpha_\ell\}\subset \mathfrak{t}^*$ the set of simple roots. The fundamental weights of $\mathfrak{g}$ are denoted by  $\{\omega_1, \dots, \omega_\ell\}\subset \mathfrak{t}^*$. Let $W$ (resp. $W_L$) be the Weyl group of $G$ (resp. $L$).
Then, $W$ is generated by the simple reflections $\{s_i\}_{1\leq i \leq \ell}$. Let $W^P$ denote the set of smallest coset representatives in the cosets in $W/W_L$. {\it Throughout the paper we follow the indexing convention as in [Bo, Planche
I - IX].} 

Let $X(T)$ be the group of characters of $T$ and let $D\subset X(T)$ be the set of dominant characters  (with respect to the given choice of $B$ and hence positive roots, which are the roots of $\mathfrak{b}$). Then, the isomorphism classes of  finite dimensional irreducible representations of $G$ are bijectively parameterized by $D$ under the correspondence $\lambda \in D \leadsto V(\lambda)$, where
$V(\lambda)$ is the irreducible representation of $G$ with highest weight $\lambda$. We call $V(\lambda)$ {\it almost faithful} if the corresponding map $\rho_\lambda: G \to \Aut (V(\lambda))$ has finite kernel. 

Recall the 
Bruhat decomposition for the flag variety:
$$G/P=\sqcup_{w\in W^P}\, \Lambda_w^P,\,\,\,\text{where}\,\, \Lambda^P_w:= BwP/P.$$
Let $\bar{\Lambda}_w^P$ denote the closure of $\Lambda_w^P$ in $G/P$.  We denote by 
$[\bar{\Lambda}_w^P] \in H_{2 \ell(w)}(G/P, \mathbb{Z})$ its
fundamental class. Let $\{\epsilon^P_w\}_{w\in W^P}$ denote the Kronecker dual basis of the cohomology, i.e., 
$$\epsilon^P_w([\bar{\Lambda}_v^P])= \delta_{w,v}, \,\,\,\text{for any}\,\, v,w\in W^P.$$
Thus, $\epsilon^P_w$ belongs to the singular cohomology:
$$\epsilon^P_w\in H^{2 \ell(w)}(G/P, \mathbb{Z}).$$
We abbreviate $\epsilon^B_w$ by $\epsilon_w$. Then, for any $w\in W^P, \epsilon^P_w=\pi^*(\epsilon_w)$, where $\pi:G/B\to G/P$ is the standard projection.

 {\it We will often abbreviate $\epsilon^P_w$ by $\epsilon_w$ when the reference to $P$ is clear from the context.}

\begin{definition}\label{def1}
Let $V(\lambda)$ be any almost faithful irreducible representation of $G$. Following Springer (cf. [BR, $\S$9]), define the map
$$
\theta_{\lambda}:G\to \mathfrak{g}\quad \text{(depending upon $\lambda$)}
$$
as follows:
\[
\xymatrix{
G\ar[r]^-{\rho_{\lambda}}\ar[dr]^{\theta_{\lambda}} & \Aut (V(\lambda))\subset \End (V(\lambda))=\mathfrak{g}\oplus \mathfrak{g}^{\perp}\ar[d]^{\pi}\\
 & \mathfrak{g}
}
\]
where  $\mathfrak{g}$ sits canonically inside $\End(V(\lambda))$ via the derivative $d\rho_{\lambda}$, the orthogonal complement $\mathfrak{g}^{\perp}$
is  taken with respect to the standard conjugate $ \Aut(V(\lambda))$-invariant form on $\End (V(\lambda))$: $\langle A, B\rangle :=\tr (AB)$,  and $\pi$ is the projection to the $\mathfrak{g}$-factor. (By considering a compact form $K$ of $G$, it is easy to see that  $ \mathfrak{g} \cap \mathfrak{g}^{\perp} =\{0\}$.)

Since $\pi\circ d\rho_{\lambda}$ is the identity map, $\theta_{\lambda}$ is a local diffeomorphism at $1$ (and hence with Zariski dense image). Of course, by construction, $\theta_{\lambda}$ is an algebraic morphism. Moreover, since the decomposition $ \End (V(\lambda))=\mathfrak{g}\oplus \mathfrak{g}^{\perp}$ is $G$-stable, it is easy to see that $\theta_\lambda$ is $G$-equivariant under conjugation.
\end{definition}
We recall the following lemma from [Ku2, Lemma 2].
\begin{lemma} \label{lemma1} The above morphism restricts to ${\theta_\lambda}_{|T}: T \to \mathfrak{t}$.
\end{lemma}

 For any $\mu \in X(T)$, we have a  $G$-equivariant
 line bundle $\cl (\mu)$  on $G/B$ associated to the principal $B$-bundle $G\to G/B$
via the one dimensional $B$-module $\mu^{-1}$. (Any  $\mu \in X(T)$ extends
 uniquely to a character of $B$.) The one dimensional $B$-module $\mu$ is also denoted by
 $\mathbb{C}_\mu$. Recall the surjective  Borel homomorphism
 $$\beta : S(\mathfrak{t}^*) \to H^*(G/B, \mathbb{C}),$$
which takes a character $ \mu \in X(T)$ to the first Chern class of the  line bundle $\cl(\mu)$. (We realize 
$X(T)$ as a lattice in $\mathfrak{t}^*$ via taking derivative.) We then extend this map linearly over $\mathbb{C}$ to $\mathfrak{t}^*$ and extend further as a graded algebra homomorphism from $ S(\mathfrak{t}^*)$ (doubling the degree). Under the Borel homomorphism,
\begin{equation}\label{eqnborel} \beta(\omega_i)=\epsilon_{s_i},\,\,\,\text{for any fundamental weight}\,\, \omega_i.
\end{equation}

Fix a compact form $K$ of $G$. In particular,  $T_o:=K\cap T$ is a (compact) maximal torus of $K$. Then, $W\simeq N(T_o)/T_o$, where $N(T_o)$ is the normalizer of $T_o$ in $K$. Recall that $\beta$ is $W$-equivariant under  the standard action of $W$ on  $S(\mathfrak{t}^*)$ and the $W$-action on $H^*(G/B, \mathbb{C})$ induced from the $W$-action on $G/B\simeq K/T_o$ via 
$$(nT_o)\cdot (kT_o):= kn^{-1}T_o,\,\,\,\text{for}\,\, n\in N(T_o)\,\,\,\text{and}\,\,  k\in K.$$
Thus, for any standard parabolic subgroup $P$ with the Levi subgroup $L$ containing $T$, restricting $\beta$, we get a surjective graded algebra homorphism:
$$\beta^P : S(\mathfrak{t}^*)^{W_L} \to H^*(G/B, \mathbb{C})^{W_L}\simeq  H^*(G/P, \mathbb{C}),$$
where the last isomorphism, which is induced from the projection $G/B \to G/P$,  can be found, e.g.,  in [Ku1, Corollary 11.3.14]. 

Now, the Springer morphism ${\theta_\lambda}_{|T}:T\to \mathfrak{t}$ (restricted to $T$) gives rise to the corresponding $W$-equivariant injective  algebra homomorphism 
on the affine coordinate rings:
$$ ({\theta_\lambda}_{|T})^*: \mathbb{C}[\mathfrak{t}]=S(\mathfrak{t}^*)\to  \mathbb{C}[T].$$
Thus,  on restriction to $W_L$-invariants, we get an injective algebra homomorphism 
$$ {\theta_\lambda (P)}^*: \mathbb{C}[\mathfrak{t}]^{W_L}=S(\mathfrak{t}^*)^{W_L}\to  \mathbb{C}[T]^{W_L}.$$
(Since $W_L$-invariants depend upon the choice of the parabolic subgroup $P$, we have included $P$ in the notation of $ {\theta_\lambda (P)}^*$.) Now, let $\Rep (L)$ be the representation ring of $L$ and let  $\Rep^\mathbb{C} (L):= \Rep (L)\otimes_{\mathbb{Z}}\,\mathbb{C}$ be its complexification.  Then, as it is well known,   
\begin{equation} \label{eqn1}\Rep^\mathbb{C} (L)\simeq  \mathbb{C}[T]^{W_L}
\end{equation}
obtained from taking the character of an $L$-module restricted to $T$. 

{\it We will often identify a virtual representation of $L$ with its character restricted to $T$ (which is automatically $W_L$-invariant).}
\begin{definition} \label{maindefi} We  call a virtual character $\chi\in  \Rep^\mathbb{C} (L)$ of $L$ a {\it $\lambda$-polynomial character} if the corresponding function in 
$ \mathbb{C}[T]^{W_L}$ is in the image of $ {\theta_\lambda (P)}^*$. The set of all $\lambda$-polynomial characters of $L$, which is, by definition,  a subalgebra of $\Rep^\mathbb{C} (L)$ isomorphic to the algebra $S(\mathfrak{t}^*)^{W_L}$, is denoted by $\Rep^\mathbb{C}_{\lambda-\poly} (L)$. Of course, the map $ {\theta_\lambda (P)}^*$ induces an algebra isomorphism (still denoted by)
$$ {\theta_\lambda (P)}^*: S(\mathfrak{t}^*)^{W_L}\simeq \Rep^\mathbb{C}_{\lambda-\poly} (L),$$
under the identification \eqref{eqn1}.
\end{definition}

It is easy to see that 
\begin{equation} \label{eqn1'} \Rep_{\omega_1-\poly}(\GL (n))= \Rep_{\poly}(\GL (n)),
\end{equation}
where  $ \Rep_{\poly}(\GL (n))$ denotes the subring of the representation ring $\Rep(\GL (n))$ spanned by the irreducible polynomial representations of $\GL (n)$. 

We recall  the following  result from [Ku2, Theorem 5].

\begin{theorem} \label{thmmain}  Let $V(\lambda)$ be an almost faithful irreducible $G$-module and let $P$ be any standard parabolic subgroup. Then,  
the above maps (specifically $\beta^P\circ ({\theta_\lambda (P)}^*)^{-1} $) give
 rise to  a surjective $\mathbb{C}$-algebra homomorphism 
$$\xi_\lambda^P: \Rep^\mathbb{C}_{\lambda-\poly}(L)  \to H^*(G/P, \C).$$

Moreover, let $Q$ be another standard  parabolic subgroup with Levi subgroup $R$ containing $T$ such that $P\subset Q$ (and hence 
$L\subset R$). Then, we have the following commutative diagram:\[
\xymatrix{
\Rep^\mathbb{C}_{\lambda-\poly}(R)  \ar[d]^{\gamma}\ar[r]^{\xi_\lambda^Q} & H^*(G/Q, \C)
\ar[d]^{\pi^*}\\
\Rep^\mathbb{C}_{\lambda-\poly}(L)\ar[r]^{\xi_\lambda^P} &H^*(G/P, \C),
}
\]
where $\pi^*$ is induced from the standard projection $\pi:G/P \to G/Q$ and $\gamma$ is induced from the restriction of representations. 
\end{theorem}

\section{Injectivity Result for the Symplectic Group}

In this section, we consider the  symplectic group $G=\Sp (2n)$ ($n\geq 2$). 
We take the Springer morphism for  $\Sp (2n)$ with respect to the first fundamental weight $\lambda= \omega_1.$ We will abbreviate the Springer morphism $\theta_{\omega_1}$ by $\theta$, $\xi_\lambda^P$ by $\xi^P$  and $\Rep^\C_{\omega_1-\poly}(G)$ by $\Rep^\C_{\poly}(G)$. 

Let $V=\mathbb{C}^{2n}$ be equipped with the
nondegenerate symplectic form $\langle \,,\,\rangle$ so that its matrix
$\bigl(\langle e_i,e_j\rangle\bigr)_{1\leq i,j \leq 2n}$ in the
standard basis $\{e_1,\dots, e_{2n}\}$ is given by
\begin{equation*}
E_C=\left(\begin{array}{cc}
0&J\\
-J&0
\end{array}\right),
\end{equation*}
where $J$ is the anti-diagonal matrix $(1,\dots,1)$ of size $n$. Let
$$\Sp(2n):=\{g\in \SL(2n):
g \,\text{leaves  the form}\, \langle \,,\,\rangle \,\text{invariant}\}$$ be the associated
symplectic group.  Clearly, $\Sp(2n)$ can be realized
as the fixed point subgroup $\SL(2n)^\sigma$ under the involution $\sigma:\SL(2n)\to \SL(2n)$
defined by $\sigma(A)=E_C(A^t)^{-1}E_C^{-1}$. 

The involution $\sigma$
keeps both of $B$ and $T$ stable, where $B$ and $T$  are the standard Borel and  maximal torus respectively of $\SL(2n)$. Moreover,
$B^\sigma$ (respectively, $T^\sigma$) is a Borel subgroup (respectively, a maximal torus)
of $\Sp(2n)$. We denote $B^\sigma, T^\sigma$ by $B_C=B_{C_n},T_C=T_{C_n}$
respectively. Then, $T_C$ is given as follows:
\begin{equation}\label{eqnnew201}
T_{C} = \left\{{\bf t}=
\diag \bigl(t_{1},  \dots, t_n, t_n^{-1}, \dots, 
 t^{-1}_{1}\bigr):t_{i}\in \mathbb{C}^{*}\right\}.
\end{equation}
Its Lie algebra is given  by 
\begin{equation}\label{eq201}
\frt_{C} =\left\{\dot{\bf t}=
\diag \bigl(x_{1},  \dots, x_n, -x_n, \dots, 
 -x_{1}\bigr):x_{i}\in \mathbb{C}\right\}.
\end{equation}
We recall the following lemma from [Ku2, Lemma 10].
\begin{lemma}\label{lem2}
The Springer morphism $\theta:G\to \mathfrak{g}$\, for  $G=\Sp (2n)$ is given by 
$$
g\mapsto \frac{g-E_C^{-1}g^{t}E_C}{2},\,\,\,\text{for}\,\, g\in G.
$$
(Observe that this is the Cayley transform.)
\end{lemma}

From the description of the Springer morphism given above, we immediately get the following (cf. [Ku2, Corollary 11]):

\begin{corollary}\label{coro3}
Restricted to the maximal torus as above, we get the following description of the Springer map $\theta$:
$$\theta({\bf t})=
\diag \bigl(\bar{t}_1
, \dots, 
  \bar{t}_n, 
  -\bar{t}_n, \dots, 
  -\bar{t}_1\bigr),\,\,\,\text{where $\bar{t}_i :=\frac{t_{i}- t^{-1}_{i}}{2}$}.$$
\end{corollary}

The following result follows easily from Corollary~\ref{coro3} together with the description of the Weyl group (cf. [Ku2, Proposition 12]).

\begin{proposition}\label{prop5}
Let   $f:T\to \mathbb{C}$  be a regular map. Then, $f\in \Rep_{\poly}^\C(G)$  if and only if the following is satisfied:

There exists a {\it symmetric} polynomial $P_f(x_{1},\ldots,x_{n})$ such that 
\begin{equation*}
f({\bf t})=P_f\left((\bar{t}_1)^2,\dots,(\bar{t}_n)^2\right), \,\,\text{for}\,\,{\bf t} \in T_C \text{ given by}\, \eqref{eqnnew201}.
\end{equation*}
\end{proposition}

We recall the following result from [Ku2, Proposition 24].
\begin{lemma}\label{lastpropo} Under the homomorphism $\xi^B:\Rep_{\poly}^\C(T) \to H^*(G/B, \C)$ of Theorem \ref{thmmain}, 
$$\bar{t}_i \mapsto  (\epsilon_{s_i}- \epsilon_{s_{i-1}}), \,\,\,\text{for any}\,\, 1\leq i\leq n .$$
\end{lemma}

\begin{definition}\label{defi9} For  $1\leq r \leq n$, we let $\IG(r,2n)$ to be
the set of $r$-dimensional isotropic
subspaces of $V$ with respect to the form $\langle\,,\,\rangle$, i.e.,
$$\IG(r,2n):=\{M\in \Gr(r,2n): \langle v,v'\rangle=0,\ \forall\,  v,v'\in M\}.$$
 Then, $\IG(r,2n)$  is the quotient $\Sp (2n)/P_r^C$ of $\Sp (2n)$ by
 the standard maximal parabolic subgroup
$P_r^C$  with $\Delta \setminus \{\alpha_r\}$ as the set of simple roots of its Levi component
$L_r^C$. (Again we take $L_r^C$ to be the unique Levi subgroup of $P_r^C$
 containing $T_C$.) 
Then, 
$$L_r^C\simeq \GL (r)\times \Sp (2(n-r)).$$
In this case,  by the identity \eqref{eqn1'}, Corollary \ref{coro3} and Proposition \ref{prop5}, 
 \begin{equation} \label{eqn201}\Rep^\mathbb{C}_{\poly}(L_r^C) \simeq \C_{\sym}[\bar{t}_1,  \dots ,  \bar{t}_r]\otimes_\C \C_{\sym}[(\bar{t}_{r+1})^2,  \dots ,  
(\bar{t}_n)^2],
\end{equation}
where $ \C_{\sym}$ denotes the subalgebra of the polynomial ring consisting of symmetric polynomials. 
\end{definition}

{\it From now on we fix $k\geq 0$ and consider  $\IG(n-k,2n)$.}

Following [BKT1, Definition 1.1],   a partition $\lambda: \lambda_1 \geq \lambda_2\geq \dots\geq \lambda_d>0$ is said to be $k$-{\it strict} if no part greater than $k$ is repeated (i.e., $\lambda_j > k \Rightarrow \lambda_{j+1} < \lambda_j$). The Schubert varieties in  $\IG(n-k,2n)$ are parametrized by 
 $k$-strict partitions contained in the $(n-k) \times (n+k)$ rectangle. The codimension of this variety is equal to $|\lambda|:= \sum\lambda_i$. Let $\sigma_\lambda
\in H^{2 |\lambda|}(\IG(n-k,2n), \Z)$ denote the cohomology class Poincar\'e dual to the fundamental class $[X_\lambda]$ of the Schubert variety associated to $\lambda$. Let $\mathcal{P}(k,n)$ denote the set of  $k$-strict partitions contained in the $(n-k) \times (n+k)$ rectangle. Thus, $\{\sigma_\lambda\}_{\lambda \in \mathcal{P}(k,n)}$ gives the Schubert basis of $ H^*(\IG(n-k,2n), \Z)$.

We have  the following short exact sequence of vector bundles over $\IG(n-k,2n)$:
$$
0 \to \mathcal{S} \to \bar{\mathcal{E}} \to \mathcal{Q} \to 0,
$$
where $ \bar{\mathcal{E}}$ is the trivial bundle of rank $2n$, $\mathcal{S}$ is the tautological subbundle of rank $n-k$ and $\mathcal{Q}$ is the  quotient bundle of rank $n+k$. Let  $c_i=c_i(\mathcal{Q})$ ($1\leq i\leq n+k$) denote  the $i^{th}$ Chern class of the quotient bundle $\mathcal{Q}$. Then, these classes are so called  the {\it special Schubert classes}. Then, by [BKT1, $\S$1.2],
\begin{equation} \label{eqn202} c_i=\sigma_{i},
\end{equation}
where $\sigma_i:=\sigma_{(i)}$ and $(i)$ is the partition with single term $i$. 

We have the following presentation of the cohomology ring due to [BKT1, Theorem 1.2]. In the following  we follow the convention that  $c_0 =1$ and $c_p =0$ if $p < 0$ or $p > n+k$.

\begin{theorem}\label{chap1-thm1.3}
  The cohomology ring $H^* (\IG(n-k, 2n), \mathbb{Z})$ is presented as a  quotient of the ploynomial ring $\mathbb{Z}[c_{1}, \ldots, c_{n+k}]$ modulo the relations:
  $$
  (R^p_{n,k})\,\, (n-k+1\leq p \leq n+k):\,\,\,\,\,\qquad\qquad \det (c_{1+ j -i})_{1 \le i , j \le p} =0,
  $$
  and
  $$
   (S^s_{n,k}) \,\,(k+1\leq s \leq n):\,\,\,\,\, \qquad\qquad c_s^2 + 2\sum^{n+k-s}_{i=1} (-1)^i c_{s+i} c_{s-i}=0.
  $$
\end{theorem}

Our original proof of the  following result was longer. The following shorter proof  is due to L. Mihalcea. 
\begin{proposition}\label{chap1-thm1.5} The map $\xi^{P_{n-k}}: \Rep^{\mathbb{C}}_{\poly} (L^C_{n-k}) \to H^* (\IG (n-k, 2n), \mathbb{C})$ of Theorem \ref{thmmain} under the decomposition  \eqref{eqn201} for $r=n-k$ takes,
  for  $1 \leq i \leq k$,
$$
e_i \left( (\bar{t}_{n-k+1})^2, \ldots , (\bar{t}_{n})^2\right) \mapsto c^2_i + 2 \sum^{i}_{j=1} (-1)^j c_{i+j}c_{i-j},
$$
where $c_i=c_i(Q)$ is as defined before Theorem \ref{chap1-thm1.3} and  $e_i$ is the $i$-th elementary symmetric function. 
\end{proposition}

Before we come to the proof of the proposition, we need the following two lemmas:
\vskip1ex

Let $\Fl = G/B$ be the full flag variety for $G=\Sp (2n)$. It consists of
partial flags 
$$ F_\bullet: \,\, F_1 \subset F_2 \subset . . . \subset F_n \subset E := \C^{2n},\,\,\text{such that each $F_j$ is isotropic and $\dim F_j=j$}.$$
We can complete the partial flag to a full flag by taking $F_{n+j}:= F_{n-j}^\perp$.
The flags  $ F_\bullet$ give rise to  a
sequence of tautological vector bundles over $\Fl$:
$$\mathcal{F}_1 \subset \mathcal{F}_2 \subset . . . \subset \mathcal{F}_n \subset \mathcal{E},\,\,\,\text{with rank  $ \mathcal{F}_j =j$},$$
where $\mathcal{E}:\Fl\times \C^{2n}\to \Fl$ is the trivial rank $2n$ vector bundle. 
For $1 \leq j\leq  n,$
define 
$$x_j:= -c_1(\mathcal{F}_j/\mathcal{F}_{j-1}),$$ 
where $\mathcal{F}_0$ is taken to be the vector bundle of rank $0$.

\begin{lemma} \label{lemma14} For $1 \leq j \leq  n$, the Schubert divisor $\epsilon_{s_j}\in H^2(\Fl, \mathbb{Z})$ is given by
$$\epsilon_{s_j} = -c_1(\mathcal{F}_j)=x_1+\dots +x_j.$$

In particular, under $\xi^B$ for $G=\Sp(2n)$, $\bar{t}_j \mapsto x_j$ for any $1\leq j\leq n$. 
\end{lemma}

\begin{proof} The first part follows from the identity \eqref{eqn202}.

The `In particular' statement follows from Lemma \ref{lastpropo}.
\end{proof}

For  $1\leq j \leq  n$, let 
$$\mathcal{Q}_j:=\mathcal{E}/\mathcal{F}_j.$$
Observe that the symplectic form gives an isomorphism of vector bundles:
\begin{equation}\label{eqn102}\mathcal{Q}_j\simeq (\mathcal{F}_j^\perp)^*.
\end{equation}

\begin{lemma} \label{newlwmma2}For  $0\leq j \leq  n$, the following holds:
$$c(\mathcal{Q}_j)c(\mathcal{Q}_j^*)=\prod_{p=j+1}^n\,(1-x_p)(1+x_p),$$
where $c$ is the total Chern class. 
\end{lemma}

\begin{proof} By definition,
\begin{equation} \label{eqn301}\prod_{p=j+1}^n\,(1-x_p)(1+x_p)= \prod_{p=j+1}^n\,c(\mathcal{F}_p/\mathcal{F}_{p-1})\cdot c\left((\mathcal{F}_p/\mathcal{F}_{p-1})^*\right)=
\frac{c(\mathcal{F}_n)}{c(\mathcal{F}_j)}\cdot \frac{c(\mathcal{F}_n^*)}{c(\mathcal{F}_j^*)}.
\end{equation}
From the exact sequence $0\to \mathcal{F}_j \to \mathcal{E}\to \mathcal{Q}_j\to 0$, we get 
\begin{equation}  \label{eqn203} 
c(\mathcal{Q}_j)\cdot c(\mathcal{F}_j) = 1\,\,\,\text{and}\,\,\, c(\mathcal{Q}_j^*)c(\mathcal{F}_j^*)=1, \end{equation}
and hence taking $j=n$ in the above equation and using the equation  \eqref{eqn102}, we get
 \begin{equation}  \label{eqn204} c(\mathcal{F}_n) c(\mathcal{F}_n^*)=1,\,\,\,\text{since $\mathcal{F}_n^\perp=\mathcal{F}_n$}.
\end{equation}  
Combining the equations \eqref{eqn301}, \eqref{eqn203} and \eqref{eqn204}, we get the lemma.
\end{proof}

\begin{proof} (of Proposition \ref{chap1-thm1.5}) 
By taking terms of degree $2i$ and $j=n-k$  in Lemma \ref{newlwmma2}, we obtain in $H^*(\Fl, \mathbb{Z})$:
$$c_i(\mathcal{Q}_j)^2+2\sum_{p=1}^i\,(-1)^pc_{i+p}(\mathcal{Q}_j)\cdot c_{i-p}(\mathcal{Q}_j)= e_i(x_{j+1}^2,\dots, x_n^2).$$
By the definition, the bundle $\mathcal{S}$ pulls back to the bundle $\mathcal{F}_{n-k}$ over  $\Fl$ under the projection $\Fl\to \IG(n-k, 2n)$.
Thus, the proposition follows from Lemma \ref{lemma14}.
\end{proof}
\begin{remark} Even though we do not need, the map $\xi^{P_{n-k}}: \Rep^{\mathbb{C}}_{\poly} (L^C_{n-k}) \to H^* (\IG (n-k, 2n), \mathbb{C})$ of Theorem \ref{thmmain} under the decomposition  \eqref{eqn201} takes
for $1\leq i \leq n-k$, 
  $$
  e_i \left(  \bar{t}_1, \ldots , \bar{t}_{n-k} \right) \mapsto c_i (S) =  \epsilon_{s_{n-k-i+1} \cdots s_{n-k}}.
  $$
This follows from [BKT1, $\S$1.2].
\end{remark}

\begin{definition}\label{defi14} [Inverse Limit]
  For any $k\geq 0$,  define the {\it stable cohomology ring} [BKT2, \S 1.3] as
  $$
  \mathbb{H}^* (\IG_k, \mathbb{Z}) = \varprojlim H^* (\IG(n-k, 2n), \mathbb{Z})
  $$
  as the inverse limit (in the category of graded rings) of the inverse system
  \begin{equation*}
  \cdots \leftarrow H^* (\IG (n-k, 2n), \mathbb{Z})\xleftarrow{\pi_n^*} H^* (\IG (n-k+1, 2n+2), \mathbb{Z}) \leftarrow \cdots ,
  \end{equation*}
where $\pi_n: \IG(n-k,  2n) \hookrightarrow \IG(n-k+1,  2n+2)$ is given by $V\mapsto T_n(V)\oplus \mathbb{C} e_{n+1}$ and $T_n:\C^{2n} \to \C^{2n+2}$ is the linear embedding taking $e_i \mapsto e_i$ for $1\leq i \leq n$ and taking $e_i \mapsto e_{i+2}$ for $n+1\leq i \leq 2n$.

  This ring has an additive basis consisting of Schubert classes $\sigma_\lambda$ for each $k-$strict partition $\lambda$. The natural ring homorphism $\varphi_{k, n}:\mathbb{H}^* (\IG_k, \mathbb{Z}) \to H^*(\IG(n-k, 2n), \mathbb{Z})$ takes $\sigma_\lambda$ to $\sigma_{\lambda}$ whenever $\lambda$ fits in a $(n-k) \times (n+k)$ rectangle and to zero otherwise. In particular, $\varphi_{k, n}$ is surjective. From the definition of the Chern classes $c_j=c_j^n(Q)$, it is easy to see that under the restriction map  $\pi_n^*: H^* (\IG (n-k+1, 2n+2), \mathbb{Z}) \to H^* (\IG (n-k, 2n), \mathbb{Z}), c_j^{n+1}\mapsto c_j^n$ for $1\leq j\leq n+k$ and $c_{n+k+1}^{n+1} \mapsto 0$. 

From the presentation of the ring $H^* (\IG(n-k, 2n), \mathbb{Z})$ (Theorem \ref{chap1-thm1.3}), none of the determinantal relations hold in the inverse limit. So, $\mathbb{H}^*(\IG_k, \mathbb{Z})$ is isomorphic to the polynomial ring $\mathbb{Z}[c_1, c_2, \ldots]$ modulo the relations:
  \begin{equation}\label{eqn205}
  (S^s) \,\,(s>k):\,\,\,\,\,\,\qquad\qquad\qquad c_s^2 + 2 \sum_{i=1}^s (-1)^i c_{s+i} c_{s-i}=0 .
  \end{equation}
  Take $k\geq 1$. Recall from Proposition \ref{prop5} that  
$$\Rep^\mathbb{C}_{\poly} (\Sp (2k)) \simeq  \mathbb{C}_{\sym} \left[(\bar{h}_1)^2 , \dots , (\bar{h}_k)^2 \right].$$
Define a ring homomorphism (for any $1\leq k\leq n$)
$$\iota^n_k: \Rep^\mathbb{C}_{\poly} (\Sp (2k)) \to \Rep^\mathbb{C}_{\poly} (L_{n-k}^C)$$
by taking $f(\bar{\bf h}) \mapsto  1 \otimes f (\bar{\bf t})$, where $\bar{\bf h}:= (\bar{h}_1, \dots , \bar{h}_k)$, 
$\bar{\bf t}:= (\bar{t}_{n-k+1}, \dots , \bar{t}_n)$ and  $f(\bar{\bf t})$ is the same polynomial written in the $\bar{\bf t}$-variables uner the transformation $\bar{h}_p \mapsto \bar{t}_{n-k+p}$. 
This gives rise to the map $\xi_{n, k}: = \xi^{P_{n-k}} \circ \iota_k^n : \Rep^{\mathbb{C}}_{\poly} (\Sp (2k)) \to H^* (\IG (n-k, 2n), \mathbb{C})$. Consider the following diagram, which is commutative because of Proposition  \ref{chap1-thm1.5}.
  \[
  \xymatrix{
     & \\
    \Rep^{\mathbb{C}}_{\poly}(\Sp(2k))\ar[r]^-{\xi_{n,k}}\ar[dr]^{\xi_{n+1,k}} & H^{*}(\IG(n-k,2n),\mathbb{C})\ar[u]^{\pi^{*}_{n-1}}\\
    & H^{*}(\IG(n-k+1,2n+2),\mathbb{C}) \ar[u]^{\pi^{*}_{n}}\\
    & \ar[u]^{\pi^{*}_{n+1}}
  }
  \]
The compatible ring homomorphisms   $\xi_{n,k}: \Rep^{\mathbb{C}}_{\poly}(\Sp(2k))\to  H^{*}(\IG(n-k,2n),\mathbb{C})$ combine to give  a ring homomorphism
$$\xi_k:  \Rep^{\mathbb{C}}_{\poly}(\Sp(2k))\to  \mathbb{H}^{*}(\IG_k,\mathbb{C}).$$ 
\end{definition}

The following theorem is one of our main results of the paper. 

\begin{theorem} \label{thmsp} Let $k\geq 1$ be an integer. 
 The above ring homomorphism  $\xi_k: \Rep^{\mathbb{C}}_{\poly} (\Sp (2k)) \to \mathbb{H}^* (\IG_k, \C)$ takes the generators
  \begin{equation} \label{eqn103}
  e_i \left((\bar{h}_1)^2, \dots , (\bar{h}_k)^2\right) \mapsto c_i^2 + 2 \sum_{j=1}^i (-1)^j c_{i+ j} c_{i-j},\,\,\,\text{for any $1\leq i \leq k$},
  \end{equation}
where $c_i=c_i(\mathcal{Q})$.

  In particular,  $\xi_k$  is injective.
\end{theorem}
\begin{proof} The first part follows  from Proposition \ref{chap1-thm1.5}. 

We next prove the injectivity of $\xi_k$:

By Proposition \ref{prop5}, $ \Rep^{\mathbb{C}}_{\poly}(\Sp (2k)) $ is a polynomial ring over $\mathbb{C}$ generated by $\{e_1(\bar{\bf h}), \dots, e_k(\bar{\bf h})\}$. Let $ \Rep^{\mathbb{Z}}_{\poly}(\Sp (2k)) $ be the polynomial subring over $\mathbb{Z}$ generated by $\{e_1(\bar{\bf h}), \dots, e_k(\bar{\bf h})\}$. Then, by the equation \eqref{eqn103}, 
$$\xi_k (\Rep^{\mathbb{Z}}_{\poly}(\Sp (2k)))\subset  \mathbb{H}^* (\IG_k, \mathbb{Z}).$$
Thus, on restriction, we get the ring homomorphism
$$\xi_k^{\mathbb{Z}}: \Rep^{\mathbb{Z}}_{\poly}(\Sp (2k)) \to  \mathbb{H}^* (\IG_k, \mathbb{Z}).$$
Observe further that $\xi_k^{\mathbb{Z}}$ is a homomorphism of graded rings if we assign degree $4i$ to each $e_i$ (and the standard cohomological degree to $ \mathbb{H}^* (\IG_k, \mathbb{Z})$). 
Let $K$ be the Kernel of $\xi_k^{\mathbb{Z}}$. Since $ \mathbb{H}^* (\IG_k, \mathbb{Z})$ is a free $\mathbb{Z}$-module of finite rank in each degree, the induced homomorphism
\begin{equation}\label{eqn104} \mathbb{Z}/(2)\otimes_\mathbb{Z} \,K \to \mathbb{Z}/(2)\otimes_\mathbb{Z}\, \Rep^{\mathbb{Z}}_{\poly}(\Sp (2k))
\,\,\,\text{is injective}. 
\end{equation}
We next observe that the induced homomorphism 
\begin{equation}\label{eqn105}  \mathbb{Z}/(2)\otimes_\mathbb{Z}\, \Rep^{\mathbb{Z}}_{\poly}(\Sp (2k))\to   \mathbb{Z}/(2)\otimes_\mathbb{Z}\,\mathbb{H}^*(\IG_k, \mathbb{Z})
\,\,\,\text{is injective}. 
\end{equation}
To prove this, observe that 
\begin{equation}\label{eqn106}  \mathbb{Z}/(2)\otimes_\mathbb{Z}\, \Rep^{\mathbb{Z}}_{\poly}(\Sp (2k))\simeq \mathbb{Z}/(2)[e_1, \dots, e_k],
\end{equation}
and, by the defining relations $(S^s) \,(s>k)$ of $ \mathbb{H}^* (\IG_k, \mathbb{Z})$ as in equation \eqref{eqn205},
\begin{equation}\label{eqn107}  \mathbb{Z}/(2)\otimes_\mathbb{Z}\, \mathbb{H}^* (\IG_k, \mathbb{Z})
\simeq \mathbb{Z}/(2)[c_1, \dots, c_k] \otimes  \frac{ \mathbb{Z}/(2)[c_{k+1}, c_{k+2}\dots ] }{\langle c_{k+1}^2, c_{k+2}^2\dots \rangle}.
\end{equation}
Moreover, under the above identifications \eqref{eqn106} and \eqref{eqn107}, by the first part of the theorem, the ring homomorphism  $\xi_k^{\mathbb{Z}}$ modulo $2$ is given by 
  $$e_i \mapsto c_i^2,\,\,\,\text{for any $1\leq i\leq k$}.$$
In particular, it is injective. From this we obtain that 
$$ \mathbb{Z}/(2)\otimes_\mathbb{Z} \,K =0.$$
But, since $K$ is a finitely generated torsionfree $\mathbb{Z}$-module in each graded degree (thus free) we get that
$$K=0.$$
Since $\mathbb{C}$ is a torsionfree $\mathbb{Z}$-module, this clearly gives the injectivity of $\xi_k$ (cf. [Sp, Chap. 5, $\S$2, Lemma 5]). This proves the theorem.
\end{proof}
\begin{remark} \label{remark1}The  ring homomorphism  $\xi_k: \Rep^{\mathbb{C}}_{\poly} (\Sp (2k)) \to \mathbb{H}^* (\IG_k, \C)$  of the above Theorem \ref{thmsp} is {\it not} surjective, as can be easily seen since the domain is a finitely generated $\mathbb{C}$-algebra (by  Proposition \ref{prop5}) whereas the range is not (for otherwise for each $n$, $H^{*}(\IG(n-k,2n),\mathbb{C})$ would be generated by a fixed finite number  of generators independent of $n$). 
\end{remark}

\section{Injectivity Result for the Odd Orthogonal Group}

The treatment in this section is parallel to that of the last section dealing with $\Sp(2n)$. But, we include some details for completeness.

In this section, we consider the  special orthogonal group  $G=\SO (2n+1)$ ($n\geq 2$). We take the Springer morphism for  $\SO (2n+1)$ with respect to the first fundamental weight $\lambda= \omega_1.$ We will abbreviate $\theta_{\omega_1}$ by $\theta$, $\xi_\lambda^P$ by $\xi^P$  and $\Rep^\C_{\omega_1-\poly}(G)$ by $\Rep^\C_{\poly}(G)$. 

Let $V'=\mathbb{C}^{2n+1}$ be equipped with the
nondegenerate symmetric form $\langle \,,\,\rangle$ so that its matrix $E_B=\
\bigl(\langle e_i,e_j\rangle\bigr)_{1\leq i,j \leq 2n+1}$ (in the standard basis
 $\{e_1,\dots, e_{2n+1}\}$) is  the $(2n+1)\times (2n+1)$
antidiagonal matrix with $1'$s all along the antidiagonal except at the $(n+1, n+1)$-th
place where the entry is $2$. Note that the associated quadratic form on $V'$ is given by
\[Q(\sum t_ie_i)= t_{n+1}^2+\sum_{i=1}^n\,t_it_{2n+2-i}.\]
 Let $$\SO(2n+1):=\{g\in \SL(2n+1):
g \,\text{leaves  the quadratic  form $Q$ invariant}\}$$ be the associated
special orthogonal group.  Clearly, $\SO(2n+1)$ can be realized
as the fixed point subgroup $\SL(2n+1)^\delta$ under the involution $\delta:\SL(2n+1)\to \SL(2n+1)$
defined by $\delta(A)=E_B^{-1}(A^t)^{-1}E_B$.
The involution $\delta$
keeps both of $B$ and $T$ stable, where $B$ (resp. $T$) is the standard Borel (resp. maximal torus) of $\SL(2n+1)$. Moreover,
$B^\delta$ (respectively, $T^\delta$) is a Borel subgroup (respectively, a maximal torus)
of $\SO(2n+1)$. We denote $B^\delta, T^\delta$ by $B_B=B_{B_n},T_B=T_{B_n}$
respectively. Then, $T_B$ is given by:
\begin{equation}\label{eqnnew202}
T_{B} = \left\{{\bf t}=
\diag \bigl(
t_{1}, \dots, t_{n}, 1, 
 t^{-1}_{n}, \dots ,  t^{-1}_{1}\bigr)
:~t_{i}\in \mathbb{C}^{*}
\right\}.
\end{equation}
Its Lie algebra is given  by 
\begin{equation} \label{eq202}
\frt_{B} = \left\{\dot{\bf t}=
\diag \bigl(
x_{1}, \dots, x_{n}, 0, 
 -x_{n}, \dots ,  -x_{1}\bigr)
:~x_{i}\in \mathbb{C}
\right\}.
\end{equation}

We recall the following lemma from [Ku2, Lemma 10].
\begin{lemma}\label{lem2'}
The Springer morphism $\theta:G\to \mathfrak{g}$\, for $G=\SO (2n+1)$ is given by 
$$
g\mapsto \frac{g-E_B^{-1}g^{t}E_B}{2},\,\,\,\text{for}\,\, g\in G.
$$
(Observe that this is the Cayley transform.)
\end{lemma}

From the description of the Springer morphism given above, we immediately get the following (cf. [Ku2, Corollary 11]:

\begin{corollary}\label{coro3'}
Restricted to the maximal torus $T_B$ as above, we get the following description of the Springer map $\theta$:
$$
\theta({\bf t})=
\diag \bigl(
\bar{t}_1, \dots, 
 \bar{t}_n, 0, 
  -\bar{t}_n, \dots, 
  -\bar{t}_1\bigr)., \,\,\,\text{where $\bar{t}_i:=\frac{t_i-t_i^{-1}}{2}$}.$$
\end{corollary}

The following result follows easily from Corollary~\ref{coro3'} together with the description of the Weyl group (cf. [Ku2, Proposition 12]).

\begin{proposition}\label{prop5'}
 Let   $f:T_B\to \mathbb{C}$  be a regular map. Then, $f\in \Rep_{\poly}^\C(G)$  if and only if the following is satisfied:

There exists a {\it symmetric} polynomial $P_f(x_{1},\ldots,x_{n})$ such that 
\begin{equation*}
f({\bf t})=P_f\left((\bar{t}_1)^2,\dots, (\bar{t}_n)^2\right), \,\,\,\text{for}\,\,
 {\bf t} \in T_B  \, \text{ given by}\, \eqref{eqnnew202}.
\end{equation*}
\end{proposition}

We recall the following result from [Ku2, Proposition 24].
\begin{lemma}\label{lastpropo'} Under the homomorphism $\xi^B:\Rep_{\poly}^\C(T_B) \to H^*(G/B, \C)$ of Theorem \ref{thmmain}
for $G=\SO(2n+1)$, 
$$ \bar{t}_i   \mapsto  (\epsilon_{s_i}- \epsilon_{s_{i-1}}), \,\,\,\text{for any}\,\, 1\leq i < n ,$$
and 
$$ \bar{t}_n \mapsto  2\epsilon_{s_n}- \epsilon_{s_{n-1}}.$$

\end{lemma}
\begin{definition}\label{defiOG} For  $1\leq r \leq n$,  let $\OG(r,2n+1)$  be
the set of $r$-dimensional isotropic
subspaces of $V'$ with respect to the quadratic form $Q$, i.e.,
$$\OG(r,2n+1):=\{M\in \Gr(r,V'): Q(v)=0,\ \forall\,  v\in M\}.$$
Then, $\OG(r,2n+1)$  is the quotient $\SO(2n+1)/P_r^B$ of $\SO(2n+1)$ by
 the standard maximal parabolic subgroup
$P_r^B$  with $\Delta \setminus \{\alpha_r\}$ as the set of simple roots of its Levi component
$L_r^B$. (Again we take $L_r^B$ to be the unique Levi subgroup of $P_r^B$
 containing $T_B$.)
Then,
\begin{equation}\label{eqn302} L_r^B\simeq \GL (r)\times \SO(2(n-r)+1).
\end{equation}
In this case,  by the identity \eqref{eqn1'}  and Proposition \ref{prop5'}, 
 $$\Rep^\mathbb{C}_{\poly}(L_r^B) \simeq \C_{\sym}[\bar{t}_1,  \dots ,  \bar{t}_r]\otimes_\C \C_{\sym}[(\bar{t}_{r+1})^2,  \dots ,  (\bar{t}_n)^2].$$
\end{definition}

{\it From now on we fix $k\geq 0$ and consider  $\OG(n-k,2n+1)$.}

 The Schubert varieties in  $\OG(n-k,2n+1)$ are again  parametrized by $\mathcal{P}(k,n)$ consisting of 
 $k$-strict partitions contained in the $(n-k) \times (n+k)$ rectangle. The codimension of this variety  is equal to $|\lambda|$. Let $\tau_\lambda
\in H^{2 |\lambda|}(\OG(n-k,2n+1), \Z)$ denote the cohomology class Poincar\'e dual to the corresponding fundamental class $[X^B_\lambda]$ of the Schubert variety associated to $\lambda$.Thus, $\{\tau_\lambda\}_{\lambda \in \mathcal{P}(k,n)}$ gives the Schubert basis of $ H^*(\OG(n-k,2n+1), \Z)$ (cf. [BKT1, $\S$2.1]).

We have  the following short exact sequence of vector bundles over $\OG(n-k,2n+1)$:
$$
0 \to \mathcal{S}_B \to \bar{\mathcal{E}}' \to \mathcal{Q}_B \to 0,
$$
where $ \bar{\mathcal{E}}'$ is the trivial bundle of rank $2n+1$, $\mathcal{S}_B$ is the tautological subbundle of rank $n-k$ and $\mathcal{Q}_B$ is the  quotient bundle of rank $n+k+1$. Let  $c_i=c_i(\mathcal{Q}_B)$ ($1\leq i\leq n+k$) denote  the $i^{th}$ Chern class of the quotient bundle $\mathcal{Q}$. (Observe that $c_{n+k+1}=0$ as can be seen by pulling $\mathcal{Q}_B$ to $\SO(2n+1)/T_B$, where it admits a nowhere vanishing section given by the vector $e_{n+1}$.) Then, by [BKT1, $\S$2.3],
\begin{equation} \label{eqn303} 
c_i (\mathcal{Q}_B) =
\begin{cases}
  \tau_i & if \quad 1\leq i \leq k\\
  2\tau_i & if \quad k <i \leq n+k ,
\end{cases}
\end{equation}
where $\tau_i:=\tau_{(i)}$ and $(i)$ is the partition with single term $i$. 

We have the following presentation of the cohomology ring due to [BKT1, Theorem 2.2(a)]. In the following  we follow the convention that  $\tau_0 =1$ and $\tau_p =0$ if $p < 0$ or $p > n+k$.

\begin{theorem}\label{chap1-thm1.3'}
  The cohomology ring $H^* (\OG(n-k, 2n+1), \mathbb{Z})$ is presented as a  quotient of the ploynomial ring $\mathbb{Z}[\tau_{1}, \ldots, \tau_{n+k}]$ modulo the relations:
  $$
  (\bar{R}^p_{n,k})\,\, (n-k+1\leq p \leq n):\,\,\,\,\,\qquad\qquad \det (\delta_{1+j-i}\tau_{1+ j -i})_{1 \le i , j \le p} =0,
  $$
 $$
  (\bar{R}^p_{n,k})\,\, (n+1\leq p \leq n+k):\,\,\,\,\,\qquad\qquad \sum_{r=k+1}^p\, (-1)^r \tau_r\det (\delta_{1+j-i}\tau_{1+ j -i})_{1 \le i , j \le p-r} =0,
  $$
  and
  $$
   (\bar{S}^s_{n,k}) \,\,(k+1\leq s \leq n):\,\,\,\,\, \qquad\qquad \tau_s^2 + \sum^{s}_{i=1} (-1)^i\delta_{s-i}\tau_{s+i} \tau_{s-i}=0,
  $$
where $\delta_p=1$ if $p\leq k$ and $\delta_p=2$ otherwise.
\end{theorem}

\begin{proposition}\label{chap1-thm1.5'} The map $\xi^{P^B_{n-k}}: \Rep^{\mathbb{C}}_{\poly} (L^B_{n-k}) \to H^* (\OG (n-k, 2n+1), \mathbb{C})$ of Theorem \ref{thmmain} under the decomposition  \eqref{eqn302}  takes,
  for  $1 \leq i \leq k$,
$$
e_i \left( (\bar{t}_{n-k+1})^2, \ldots , (\bar{t}_{n})^2\right) \mapsto c^2_i + 2 \sum^{i}_{j=1} (-1)^j c_{i+j}c_{i-j},
$$
where $c_i=c_i(\mathcal{Q}_B)$  and  $e_i$ is the $i$-th elementary symmetric function. 
\end{proposition}

\begin{proof} It follows by the same proof as that of the corresponding Propositionn \ref{chap1-thm1.5} once we use the following two lemmas.
\end{proof} 

Let $\Fl_B = G/B_B$ be the full flag variety for $G=\SO (2n+1)$. It consists of
partial flags 
$$ \bar{F}_\bullet: \,\, \bar{F}_1 \subset \bar{F}_2 \subset . . . \subset \bar{F}_n \subset E' := \C^{2n+1},\,\,\text{such that each $\bar{F}_j$ is istropic and $\dim F_j=j$}.$$
We can complete the partial flag to a full flag by taking $\bar{F}_{n+j}:= \bar{F}_{n-j}^\perp$.
The flags  $ \bar{F}_\bullet$ give rise to  a
sequence of tautological vector bundles over $\Fl_B$:
$$\bar{\mathcal{F}}_1 \subset \bar{\mathcal{F}}_2 \subset . . . \subset \bar{\mathcal{F}}_n \subset \mathcal{E}',\,\,\,\text{with rank  $ \bar{\mathcal{F}}_j =j$},$$
where $\mathcal{E}':\Fl_B\times \C^{2n+1}\to \Fl_B$ is the trivial rank $2n+1$ vector bundle. 
For $1 \leq j\leq  n,$
define 
$$\bar{x}_j:= -c_1(\bar{\mathcal{F}}_j/\bar{\mathcal{F}}_{j-1}),$$ 
where $\bar{\mathcal{F}}_0$ is taken to be the vector bundle of rank $0$.

The first part of the following lemma follows from equation \eqref{eqn303}. The `In particular' statement follows from Lemma \ref{lastpropo'}.
\begin{lemma} For $1 \leq j \leq  n$, the Schubert divisor $\epsilon_{s_j}\in H^2(\Fl_B, \mathbb{Z})$ is given by
$$\epsilon_{s_j} = -c_1(\bar{\mathcal{F}}_j)=\bar{x}_1+\dots +\bar{x}_j,\,\,\,\text{for $j<n$, \,\,and }$$
$$2\epsilon_{s_n} = -c_1(\bar{\mathcal{F}}_n)=\bar{x}_1+\dots +\bar{x}_n.$$

In particular, under $\xi^B$ for $G=\SO(2n+1)$, $\bar{t}_j \mapsto \bar{x}_j$ for any $1\leq j\leq n$. 
\end{lemma}

For  $1\leq j \leq  n$, let 
$$\bar{\mathcal{Q}}_j:=\mathcal{E}'/\bar{\mathcal{F}}_j .$$
Observe that the orthogonal form gives an isomorphism of vector bundles:
\begin{equation}\label{eqn102'}\bar{\mathcal{Q}}_j\simeq (\bar{\mathcal{F}}_j^\perp)^*.
\end{equation}

\begin{lemma} \label{newlwmma2'}For  $0\leq j \leq  n$, the following holds:
$$c(\bar{\mathcal{Q}}_j)c(\bar{\mathcal{Q}}_j^*)=\prod_{p=j+1}^n\,(1-\bar{x}_p)(1+\bar{x}_p),$$
where $c$ is the total Chern class. 
\end{lemma}

\begin{proof} The  lemma follows by the same proof as that of the corresponding Lemma \ref{newlwmma2} once we observe that 
$$c(\bar{\mathcal{F}}_n)= c(\bar{\mathcal{F}}_{n+1}),$$
which follows from the fact that $\bar{\mathcal{F}}_{n+1}/\bar{\mathcal{F}}_{n}$ pulled back to $\SO(2n+1)/T_B$ admits a nowhere vanishing section 
since the vector $e_{n+1}$ is held fixed by $T_B$. 
\end{proof}
\begin{remark} \label{remark2}Even though we do not need, the map $\xi^{P^B_{n-k}}: \Rep^{\mathbb{C}}_{\poly} (L^B_{n-k}) \to H^* (\OG (n-k, 2n+1), \mathbb{C})$ of Theorem \ref{thmmain} under the decomposition  \eqref{eqn302} takes
for $1\leq i \leq n-k$, 
  $$
  e_i \left(  \bar{t}_1, \ldots , \bar{t}_{n-k} \right) \mapsto c_i (S_B) =  \epsilon_{s_{n-k-i+1} \cdots s_{n-k}}, \,\,\,\text{if $k>0$,}
  $$
 $$
  e_i \left(  \bar{t}_1, \ldots , \bar{t}_{n-k} \right) \mapsto c_i (S_B) =  2\epsilon_{s_{n-k-i+1} \cdots s_{n-k}}, \,\,\,\text{if $k=0$.}
  $$
\end{remark}

\begin{definition}\label{defi14'} [Inverse Limit]
  Analogous to Definition \ref{defi14}, for any $k\geq 0$,  define the {\it stable cohomology ring} [BKT2, \S 3.2] as
  $$
  \mathbb{H}^* (\OG_k, \mathbb{Z}) = \varprojlim H^* (\OG(n-k, 2n+1), \mathbb{Z})
  $$
  as the inverse limit (in the category of graded rings) of the inverse system
  \begin{equation*}
  \cdots \leftarrow H^* (\OG (n-k, 2n+1), \mathbb{Z})\xleftarrow{\bar{\pi}_n^*} H^* (\OG (n-k+1, 2n+3), \mathbb{Z}) \leftarrow \cdots ,
  \end{equation*}
where $\bar{\pi}_n: \OG(n-k,  2n+1) \hookrightarrow \OG(n-k+1,  2n+3)$ is given by $V\mapsto \bar{T}_n(V)\oplus \mathbb{C} e_{n+1}$ and $\bar{T}_n:\C^{2n+1} \to \C^{2n+3}$ is the linear embedding taking $e_i \mapsto e_i$ for $1\leq i \leq n$, taking  $e_{n+1} \mapsto e_{n+2}$ and taking $e_i \mapsto e_{i+2}$ for $n+2\leq i \leq 2n+1$.

  This ring has an additive basis consisting of Schubert classes $\tau_\lambda$ for each $k-$strict partition $\lambda$. The natural ring homorphism $\bar{\varphi}_{k, n}:\mathbb{H}^* (\OG_k, \mathbb{Z}) \to H^*(\OG(n-k, 2n+1), \mathbb{Z})$ takes $\tau_\lambda$ to $\tau_{\lambda}$ whenever $\lambda$ fits in a $(n-k) \times (n+k)$ rectangle and to zero otherwise. In particular, $\bar{\varphi}_{k, n}$ is surjective. From the definition of the Chern classes $c_j=c_j^n(Q_B)$, it is easy to see that under the restriction map  $\bar{\pi}_n^*: H^* (\OG (n-k+1, 2n+3), \mathbb{Z}) \to H^* (\OG (n-k, 2n+1), \mathbb{Z}), c_j^{n+1}\mapsto c_j^n$ for $1\leq j\leq n+k$ and $c_{n+k+1}^{n+1} \mapsto 0$. 

From the presentation of the ring $H^* (\OG(n-k, 2n+1), \mathbb{Z})$ (Theorem \ref{chap1-thm1.3'}), $\mathbb{H}^*(\OG_k, \mathbb{Z})$ is isomorphic with the polynomial ring $\mathbb{Z}[\tau_1, \tau_2, \ldots]$ modulo the relations:
  \begin{equation}\label{eqn305}
  (\bar{S}^s) \,\,(s>k):\,\,\,\,\,\,\qquad\qquad\qquad \tau_s^2 + \sum_{i=1}^s (-1)^i \delta_{s-i}\tau_{s+i} \tau_{s-i}=0 .
  \end{equation}

  Take $k\geq 1$. Recall from Proposition \ref{prop5'} that  
$$\Rep^\mathbb{C}_{\poly} (\SO (2k+1)) \simeq  \mathbb{C}_{\sym} \left[(\bar{h}_1)^2 , \dots , (\bar{h}_k)^2 \right].$$
Define a ring homomorphism (for any $1\leq k\leq n$)
$$\bar{\iota}^n_k: \Rep^\mathbb{C}_{\poly} (\SO (2k+1)) \to \Rep^\mathbb{C}_{\poly} (L_{n-k}^B)$$
by taking $f(\bar{\bf h}) \mapsto  1 \otimes f (\bar{\bf t})$, where $\bar{\bf h}:= (\bar{h}_1, \dots , \bar{h}_k)$, 
$\bar{\bf t}:= (\bar{t}_{n-k+1}, \dots , \bar{t}_n)$ and  $f(\bar{\bf t})$ is the same polynomial written in the $\bar{\bf t}$-variables uner the transformation $\bar{h}_p \mapsto \bar{t}_{n-k+p}$. 
This gives rise to the map $\bar{\xi}_{n, k}: = \xi^{P_{n-k}^B} \circ \bar{\iota}_k^n : \Rep^{\mathbb{C}}_{\poly} (\SO (2k+1)) \to H^* (\OG (n-k, 2n+1), \mathbb{C})$. Consider the following diagram, which is commutative because of Proposition  \ref{chap1-thm1.5'}.
  \[
  \xymatrix{
     & \\
    \Rep^{\mathbb{C}}_{\poly}(\SO(2k+1))\ar[r]^-{\bar{\xi}_{n,k}}\ar[dr]^{\bar{\xi}_{n+1,k}} & H^{*}(\OG(n-k,2n+1),\mathbb{C})\ar[u]^{\bar{\pi}^{*}_{n-1}}\\
    & H^{*}(\OG(n-k+1,2n+3),\mathbb{C}) \ar[u]^{\bar{\pi}^{*}_{n}}\\
    & \ar[u]^{\bar{\pi}^{*}_{n+1}}
  }
  \]
The compatible ring homomorphisms   $\bar{\xi}_{n,k}: \Rep^{\mathbb{C}}_{\poly}(\SO(2k+1))\to  H^{*}(\OG(n-k,2n+1),\mathbb{C})$ combine to give  a ring homomorphism
$$\bar{\xi}_k:  \Rep^{\mathbb{C}}_{\poly}(\SO(2k+1))\to  \mathbb{H}^{*}(\OG_k,\mathbb{C}).$$ 
\end{definition}

The following theorem is our second main result of the paper, which is analogous to Theorem \ref{thmsp}. 

\begin{theorem} \label{thmsp'}  Let $k\geq 1$ be an integer. 
 The above ring homomorphism  $\bar{\xi}_k: \Rep^{\mathbb{C}}_{\poly} (\SO (2k+1)) \to \mathbb{H}^* (\OG_k, \C)$ takes the generators
  \begin{equation} \label{eqn103'}
  e_i \left((\bar{h}_1)^2, \dots , (\bar{h}_k)^2\right) \mapsto c_i^2 + 2 \sum_{j=1}^i (-1)^j c_{i+ j} c_{i-j},\,\,\,\text{for any $1\leq i \leq k$},
  \end{equation}
where $c_i:=c_i(\mathcal{Q}_B)$. 

  In particular,  $\bar{\xi}_k$  is injective.
\end{theorem}
\begin{proof} The first part follows  from Proposition \ref{chap1-thm1.5'}. 

We next prove the injectivity of $\bar{\xi}_k$:

By Proposition \ref{prop5'}, $ \Rep^{\mathbb{C}}_{\poly}(\SO (2k+1)) $ is a polynomial ring over $\mathbb{C}$ generated by $\{e_1(\bar{\bf h}), \dots, e_k(\bar{\bf h})\}$. Let $ \Rep^{\mathbb{Z}}_{\poly}(\SO (2k+1)) $ be the polynomial subring over $\mathbb{Z}$ generated by $\{e_1(\bar{\bf h}), \dots, e_k(\bar{\bf h})\}$. 

Let $ \bar{\mathbb{H}}^* (\OG_k, \mathbb{Z})\subset  \mathbb{H}^* (\OG_k, \mathbb{Z})$ be the subring generated by $\{c_i\}_{i\geq 1}$. Then, by the identity \eqref{eqn303},
\begin{equation}\label{eqn501} 
 \C\otimes_{\mathbb{Z}}\,\bar{\mathbb{H}}^* (\OG_k, \mathbb{Z})= \C\otimes_{\mathbb{Z}}\,\mathbb{H}^* (\OG_k, \mathbb{Z})=  \mathbb{H}^* (\OG_k, \mathbb{C}).
 \end{equation}
 Then, by the equation \eqref{eqn103'}, 
$$\bar{\xi}_k (\Rep^{\mathbb{Z}}_{\poly}(\SO (2k+1)))\subset  \bar{\mathbb{H}}^* (\OG_k, \mathbb{Z}).$$
Thus, on restriction, we get the ring homomorphism
$$\bar{\xi}_k^{\mathbb{Z}}: \Rep^{\mathbb{Z}}_{\poly}(\SO (2k+1)) \to  \bar{\mathbb{H}}^* (\OG_k, \mathbb{Z}).$$
Observe further that $\bar{\xi}_k^{\mathbb{Z}}$ is a homomorphism of graded rings if we assign degree $4i$ to each $e_i$ (and the standard cohomological degree to $ \bar{\mathbb{H}}^* (\OG_k, \mathbb{Z})$). 
Let $\bar{K}$ be the Kernel of $\bar{\xi}_k^{\mathbb{Z}}$. Since $ \mathbb{H}^* (\OG_k, \mathbb{Z})$ is a free $\mathbb{Z}$-module of finite rank in each degree
and hence so is $\bar{\mathbb{H}}^* (\OG_k, \mathbb{Z})$, the induced homomorphism
\begin{equation}\label{eqn104} \mathbb{Z}/(2)\otimes_\mathbb{Z} \,\bar{K} \to \mathbb{Z}/(2)\otimes_\mathbb{Z}\, \Rep^{\mathbb{Z}}_{\poly}(\SO (2k+1))
\,\,\,\text{is injective}. 
\end{equation}
We next observe that the induced homomorphism 
\begin{equation}\label{eqn105}  \mathbb{Z}/(2)\otimes_\mathbb{Z}\, \Rep^{\mathbb{Z}}_{\poly}(\SO (2k+1))\to   \mathbb{Z}/(2)\otimes_\mathbb{Z}\,\bar{\mathbb{H}}^*(\OG_k, \mathbb{Z})
\,\,\,\text{is injective}. 
\end{equation}
To prove this, observe that 
\begin{equation}\label{eqn106'}  \mathbb{Z}/(2)\otimes_\mathbb{Z}\, \Rep^{\mathbb{Z}}_{\poly}(\SO (2k+1))\simeq \mathbb{Z}/(2)[e_1, \dots, e_k].
\end{equation}
Moreover, by the defining relations $(\bar{S}^s) \,(s>k)$ of $ \mathbb{H}^* (\OG_k, \mathbb{Z})$ as in equation \eqref{eqn305} together with the identity \eqref{eqn303},
we can rewrite the equation \eqref{eqn305} as:
$$ (\hat{S}^s) \,\,(s>k):\,\,\,\,\,\,\qquad\qquad\qquad c_s^2 + 2\sum_{i=1}^s (-1)^i c_{s+i} c_{s-i}=0.$$
Thus, 
\begin{equation}\label{eqn107'}  \mathbb{Z}/(2)\otimes_\mathbb{Z}\, \bar{\mathbb{H}}^* (\OG_k, \mathbb{Z})
\simeq \mathbb{Z}/(2)[c_1, \dots, c_k] \otimes  \frac{ \mathbb{Z}/(2)[c_{k+1}, c_{k+2}\dots ] }{\langle c_{k+1}^2, c_{k+2}^2\dots \rangle}.
\end{equation}
Moreover, under the above identifications \eqref{eqn106'} and \eqref{eqn107'}, by the first part of the theorem,  the ring homomorphism  $\bar{\xi}_k^{\mathbb{Z}}$ modulo $2$ is given by 
  $$e_i \mapsto c_i^2,\,\,\,\text{for any $1\leq i\leq k$}.$$
In particular, it is injective. From this we obtain that 
$$ \mathbb{Z}/(2)\otimes_\mathbb{Z} \,\bar{K} =0.$$
But, since $\bar{K}$ is a finitely generated torsionfree $\mathbb{Z}$-module in each graded degree (thus free) we get that
$$\bar{K}=0.$$
Since $\mathbb{C}$ is a torsionfree $\mathbb{Z}$-module, by the equation \eqref{eqn501}, this clearly gives the injectivity of $\bar{\xi}_k$  proving the theorem.
\end{proof}

\begin{remark} \label{remark2}The  ring homomorphism  $\bar{\xi}_k: \Rep^{\mathbb{C}}_{\poly} (\SO(2k+1)) \to \mathbb{H}^* (\OG_k, \C)$  of the above Theorem \ref{thmsp'} is {\it not} surjective, as can be easily seen since the domain is a finitely generated $\mathbb{C}$-algebra (by  Proposition \ref{prop5'}) whereas the range is not (for otherwise for each $n$, $H^{*}(\OG(n-k,2n+1),\mathbb{C})$ would be generated by a fixed finite number  of generators independent of $n$). 
\end{remark}

\bibliographystyle{plain}

\begin{thebibliography}{10}



\bibitem[BR] {BR} P. Bardsley and R.W. Richardson, \'{E}tale slices for algebraic transformation groups in characteristic $p$,
{\em Proc. London Math. Soc.} {\bf 51} (1985), 295--317.


 \bibitem[Bo]{Bourbaki}
N. Bourbaki, {\em Groupes et Alg\`ebres de Lie}, Chap. 4--6,
Masson, Paris, 1981.

\bibitem[F]{fulton1}
W. Fulton,
{\em Young Tableaux},
 London Math. Society, Cambridge University Press, 1997.



\bibitem[Ku1]{Kumar1}
S.  Kumar, {\em Kac-Moody Groups, their Flag Varieties and
Representation Theory}, Progress in Mathematics, vol. {\bf 204},
Birkh\"auser, 2002.

\bibitem[Ku2]{Kumar2}
S.  Kumar, 
Representation ring of Levi subgroups versus cohomology ring of flag varieties,
{\it Math. Annalen} {\bf 366} (2016), 395 --415. 

\bibitem[BKT1] {BKT1} A. Buch, A. Kresch and H. Tamvakis, Quantum Pieri rules for isotropic Grassmannians, {\it Invent.  Math.} {\bf 178} (2009), 345--405.


\bibitem[BKT2] {BKT2} A. Buch, A. Kresch and H. Tamvakis, Quantum  Giambelli formulas  for isotropic Grassmannians, {\it  Math. Annalen} {\bf 354} (2012), 801--812.

\bibitem[Sp]{Sp}
E.  H. Spanier, {\em Algebraic Topology}, McGraw-Hill, 1966.

\end{thebibliography}
\def\noopsort#1{}

\vskip5ex

\noindent
Address: Shrawan Kumar,
Department of Mathematics,
University of North Carolina,
Chapel Hill, NC  27599--3250. 
\noindent
email: shrawan@email.unc.edu
\vskip1ex

Sean Rogers,
NSA.

\end{document}